\documentclass[10pt]{amsart}
\usepackage{graphicx,amsmath,epstopdf,xypic,somedefs,tracefnt,latexsym,rawfonts,amsfonts,amssymb,latexsym,enumerate,amscd}
\def\cal{\mathcal}
\def\Bbb{\mathbb}

\def\r{\rangle}
\def\l{\langle}

\def\vtr{\trianglelefteq}
\newtheorem{prop}{Proposition}[section]
\newtheorem{thm}{Theorem}[section]

\newtheorem{lemma}{Lemma}[section]
\newtheorem{slemma}{Sublemma}[section]
\newtheorem{cor}{Corollary}[section]

\newtheorem{rem}{Remark}[section]
\numberwithin{equation}{section}
\begin{document}
\title{The affine Artin group of type $\widetilde B_n$ is virtually poly-free}
\author[Li Li]{Li Li}
\address{Department of Mathematics and Statistics\\
Oakland University\\
Rochester, MI 48309-4479, USA}
\email{li2345@oakland.edu} 
\author[S.K. Roushon]{S.K. Roushon}
\address{School of Mathematics\\
Tata Institute\\
Homi Bhabha Road\\
Mumbai 400005, India}
\email{roushon@math.tifr.res.in} 
\urladdr{http://mathweb.tifr.res.in/\~\ roushon/}
\date{\today}
\begin{abstract}
In this note we prove that the affine Artin group of type
$\widetilde B_n$ is virtually poly-free. The proof also gives another
solution of the $K(\pi, 1)$ problem for $\widetilde B_n$.\end{abstract}


\maketitle

\section{Introduction}

Let $K=\{s_1,s_2,\dots,s_k\}$ be a finite set, and $m:K\times K\to
\{1,2,\dots, \infty\}$ be a 
map with the property that $m(s,s)=1$, and $m(s',s)=m(s,s')\geq 2$ for $s\neq s'$. 
The {\it Coxeter group} associated with the pair $(K,m)$ is, by definition, the 
following group.
$${\cal W}_{(K,m)}=\langle K\ |\ (ss')^{m(s,s')}=1,\ s,s'\in K\ \text{and}\ m(s,s')<\infty\rangle.$$
A complete 
classification of finite, irreducible Coxeter groups is known
(\cite{Cox}).
Finite  Coxeter groups are exactly the finite reflection groups.
Also, there are infinite Coxeter groups which are affine reflection groups (\cite{Hum}). 

The {\it Artin group} associated with the Coxeter
group ${\cal W}_{(K, m)}$ is, by definition, the following group.
$${\cal A}_{(K, m)}=\l K\ |\ ss'ss'\cdots = s'ss's\cdots,\ s,s'\in K\r.$$
Here, the number of times the factors in $ss'ss'\cdots$ appear is 
$m(s,s')$; e.g., if $m(s,s')=3$, then the relation is $ss's=s'ss'$. 
${\cal A}_{(K,m)}$ is called the {\it Artin group of type 
${\cal W}_{(K,  m)}$}.

A {\it finite type} or an {\it affine type} Artin group is 
the Artin group associated with a finite or an affine type Coxeter group, respectively.
There are {\it complex type} Artin groups also, which are the Artin
groups whose corresponding Coxeter groups are generated by reflections
along complex hyperplanes in some complex space.

For some details on this
subject see \cite{Hum}, \cite{Bri1} and \cite{Br}.

A group $G$ is
    called {\it poly-free}, if $G$ 
    admits a normal series $1=G_0\vtr G_1\vtr G_2\vtr \cdots \vtr G_n=G$, such that 
    $G_{i+1}/G_i$ is free, for $i=0,1,\dots, n-1$. $G$ is called
    {\it virtually poly-free} if $G$ contains a finite index poly-free 
    subgroup.

Poly-free groups have nice properties like, locally indicable and right orderable. Also, 
an inductive argument using [\cite{DS}, Theorem 2.3] shows that 
a virtually poly-free group has finite asymptotic dimension.

It is still an open question if all Artin groups are virtually poly-free. See
[\cite{MB}, Question 2].
Some of the finite type Artin groups are already known to be virtually
poly-free (\cite{Br}).
It was shown  
in \cite{BMP} that the even Artin groups (that is when $m(s,s')=2$ for all $s\neq s'$)
of $FC$-types (certain amalgamation of finite type even Artin groups) are poly-free.
A simple proof of this result of \cite{BMP} is given in \cite{W}.

We extended this class and proved in [\cite{Rou20}, Theorem 2.19], that an Artin group of the affine type
  $\tilde A_n$, $\tilde B_n$, $\tilde C_n$, $\tilde D_n$ or of the finite
  complex type $G(de,e,r)$ ($d,r\geq 2$) is virtually poly-free.

However, in \cite{JF23}, an error was pointed out in \cite{Rou20}. Therefore,
the statement (\cite{Rou20}, Theorem 2.19) 
that the affine Artin groups
${\cal A}_{\widetilde B_n}$ and ${\cal A}_{\widetilde D_n}$ of types 
$\widetilde B_n$ and $\widetilde D_n$, respectively, are
virtually poly-free, needs a new proof.

In this article we give an argument to settle the claim in the
$\widetilde B_n$ case. The proof involves the construction of a fibration
between two hyperplane arrangement complements in the
complex space (Proposition \ref{fibration}), which also
gives another solution of the $K(\pi, 1)$ problem
for ${\cal A}_{\widetilde B_n}$ (Corollary \ref{kpione}). See \cite{CMS10}. 

\begin{thm}\label{mt} ${\cal A}_{\widetilde B_n}$
  ($n\geq 3$) is virtually poly-free.\end{thm}

\begin{proof}
Consider the following complement in ${\Bbb C}^n$ of the
reflecting hyperplanes in the case of the affine reflection
group $W_{\widetilde B_n}$ of type $\widetilde
B_n$ (\cite{All02}, Section 4).
$$M_{\widetilde B_n}=\{u\in{\Bbb C}^n\ |\ u_i\pm u_j\notin {\Bbb Z},\ \text{for}\
i\neq j; u_k\notin {\Bbb Z},\ \text{for}\ k=1,2,\dots, n\}.$$

Recall that $W_{\widetilde B_n}$ acts properly
discontinuously on $M_{\widetilde B_n}$ and ${\cal A}_{\widetilde B_n}\simeq
\pi_1(M_{\widetilde B_n}/W_{\widetilde B_n})$. Hence, we have the following short exact sequence
(\cite{Ngu83}, Section 5).

\centerline{
  \xymatrix{1\ar[r]&\pi_1(M_{\widetilde B_n})\ar[r]&{\cal A}_{\widetilde
      B_n}\ar[r]&W_{\widetilde B_n}\ar[r]&1.}}

Next, note that an affine reflection group is an extension
of the corresponding finite reflection group, and a
finitely generated free abelian group. That is, $W_{\widetilde B_n}$
contains a finitely generated free abelian group
of finite index (\cite{All02}, Section 4).

Therefore, it is enough to prove that $\pi_1(M_{\widetilde B_n})$ is poly-free, since
an extension of a virtually poly-free group by a poly-free group is
virtually poly-free.

Consider the exponential covering map $exp:{\Bbb C}\to {\Bbb
  C^*}:={\Bbb C}-\{0\}$, $z\mapsto e^{2\pi iz}$. Then $exp$ induces the following covering map.
$$M_{\widetilde B_n}\to N:=\{v\in{\Bbb C}^n\ |\ v_i\neq v_j^{\pm 1},\ \text{for}\
i\neq j; v_k\neq 0,1,\ \text{for}\ k=1,2,\dots, n\}.$$

Consequently, it is enough to prove that $\pi_1(N)$ is poly-free, since
subgroup of a poly-free group is poly-free.

Now, we consider the map $\eta:{\Bbb C}-\{0,1\}\to {\Bbb C}-\{\pm 1\}$ defined by  
$\alpha\mapsto \frac{\alpha+1}{\alpha-1}$. This map induces the following homeomorphism.
$$N\to P:=\{w\in{\Bbb C}^n\ |\ w_i\neq \pm w_j,\ \text{for}\
i\neq j; w_k\neq \pm 1,\ \text{for}\ k=1,2,\dots, n\}.$$

Next, consider the following homeomorphism.
$$P\times {\Bbb C}^*\to \{y\in {\Bbb C}^{n+1}\ |\ y_i\neq \pm y_j,\ \text{for}\
i\neq j; y_1\neq 0 \}.$$
$$(w_1,w_2,\dots w_n,\lambda)\mapsto (\lambda, \lambda w_1,\lambda
w_2,\dots \lambda w_n).$$

We will show that the fundamental group of the range space 
is poly-free, which is the content of Proposition \ref{fibration}. Hence, $\pi_1(P)$
is poly-free.

This completes the proof of the theorem.
\end{proof}

\begin{rem}{\rm Together with \cite{BFW23}, Theorem \ref{mt} implies 
    that the Farrell-Jones Isomorphism conjecture is true for ${\cal
      A}_{\widetilde B_n}$. For an alternate proof of the conjecture for
    ${\cal A}_{\widetilde B_n}$, see the Corrigendum in
    \cite{Rou22}.}\end{rem}

\noindent
{\bf Acknowledgement.} The authors are grateful to the referee for
reading the paper carefully and for pointing out the reference \cite{ADR20}.

\section{A fibration}
In this section we state and prove the proposition
we referred to in the Introduction.

\begin{prop}\label{fibration}
  Consider the following two spaces.
$$Y=\{y=(y_1,y_2,\dots, y_n)\in {\Bbb C}^n\ |\ y_i\neq \pm y_j,\ \text{for}\ 
  i\neq j; y_1\neq 0 \}$$ and 
  $$Z=\{z=(z_1,z_2,\dots, z_{n-1})\in ({\Bbb C}^*)^{n-1}\ |\ z_i\neq z_j,\ \text{for}\ i\neq
  j\}.$$ 
Then the map  
$$f:Y\to Z, \; (y_1,y_2,\cdots, y_n)\mapsto (y_1(y_1^2-y_n^2),\dots, y_1(y_{n-1}^2-y_n^2))$$
  is a locally
  trivial fibration, with non-compact connected fibers homeomorphic
  to a $2$-manifold of genus $(3n-6)2^{n-3}+1$, with $3\cdot 2^{n-2}$ points removed. 

Consequently, $\pi_1(Y)$ is poly-free.
\end{prop}

\begin{rem}{\rm The referee pointed out to us that
Proposition \ref{fibration} is
also a special case of the main result in \cite{ADR20}, which
we were not aware of at the time of writing this paper. However, our proof
differs from the proof given in \cite{ADR20}.}\end{rem}

\begin{cor}\label{kpione} (\cite{CMS10}) $M_{\widetilde B_n}$ is aspherical.\end{cor}
\begin{proof} From Proposition \ref{fibration}, we get that $Y$ is aspherical, since
  it fibers over an aspherical space with aspherical fiber. And hence, $P$
  is aspherical, since $P\times {\Bbb C}^*$ is homeomorphic to $Y$ (for $n+1$).
  Therefore, $N$ is aspherical, as $N$ is homeomorphic to $P$. Consequently, $M_{\widetilde B_n}$ is
  aspherical, because $M_{\widetilde B_n}$ is a covering space of $N$.\end{proof}

In [\cite{CMS10}, Lemma 3.1] it was proved that $Y$ is a simplicial arrangement. Then, 
since by \cite{Del}, a simplicial arrangement is aspherical, so is $Y$.

Proposition \ref{fibration} is an application of Thom's First
Isotopy Lemma (\cite{Tho69} and \cite{Mat71}), and the
Fadell-Neuwirth Fibration Theorem (\cite{FN62}, Theorem 3).

\begin{thm} (Thom's First Isotopy Lemma) Let $f:Y\to Z$ be a smooth
  map between two smooth manifolds $Y$ and $Z$. Let $A\subset Y$ be a
closed Whitney stratified subset. Let $f|_A$ be proper and $f|_S$ is
a submersion for every stratum $S$ of $A$. Then, $f|_A$ is a locally
trivial fibration. Moreover, $f|_S$ is a locally trivial
fibration for any stratum $S$ of $A$.\end{thm}

\begin{thm} (Fadell-Neuwirth Fibration Theorem) Let $X$ be a connected
  manifold of dimension $\geq 2$, and $n\geq 2$. Then the projection map $X^n\to X^{n-1}$
  to the first $n-1$ coordinates, restricts to the following locally trivial fibration,
  with fiber homeomorphic to $X-\{(n-1)\text{-points}\}$. 

  $$\{x\in X^n\ |\ x_i\neq x_j,\ \text{for}\ i\neq j\}\to \{x\in X^{n-1}\ |\ x_i\neq x_j,\ \text{for}\ i\neq
  j\}.$$
    \end{thm}

  Consider the graph of the map $f:Y\to Z$ in ${\Bbb C}^n\times Z$, and the following 
  composite map.

  $$\eta:Y\to {\Bbb C}^n\times Z \subset {\Bbb P}^n\times Z.$$
  $$(y_1,y_2,\dots y_n)\mapsto ((1:y_1:y_2:\dots: y_n),
  f(y_1,y_2,\dots, y_n)).$$

  Let $C=\eta(Y)$ and $\pi:{\Bbb P}^n\times Z\to Z$ be the second
  projection. 

  We denote the closure of $C$ in ${\Bbb P}^n\times Z$ by
  $\overline C$ (in the sense of the usual topology). Let $C_z$ and
  $\overline C_z$ be the fibers of $\pi|_C$ and $\pi|_{\overline C}$
  over $z\in Z$, respectively.

  Let $S$ denote the following set of polynomials.
  $$S=\{z_i-y_1(y_i^2-y_n^2),\ \text{for}\ i=1,2,\dots, n-1\}.$$
  
  We know the zero set of $S$ in ${\Bbb C}^n\times Z$ is equal to $C$.

  Consider the homogenization $S_h$ of $S$ as follows.
  $$S_h=\{z_iy_0^3-y_1(y_i^2-y_n^2),\ \text{for}\ i=1,2,\dots, n-1\}.$$
  
  Then the zero set ${\mathcal Z}(S_h)$ of $S_h$ in ${\Bbb P}^n\times Z$ contains not
  only the closure $\overline C$ of $C$, but also another irreducible component defined
  by $y_0=y_1=0$. So we need to introduce more equations to define $\overline{C}$.
To do this, note that all the points of $\overline C$ satisfy the
following equations as well:
$$E=\Biggl\{\frac{y_i^2-y_n^2}{z_i}=\frac{y_j^2-y_n^2}{z_j},\
\text{for}\ i,j=1,2,\dots, n-1, i < j\Biggr\}$$
or equivalently, satisfy the following polynomials:
$$T=\{(y_1^2-y_n^2)z_j-(y_j^2-y_n^2)z_1,\
\text{for}\ j=2,\dots, n-1\}.$$

\begin{lemma}\label{lemma B12}
  $\overline C={\mathcal Z}(S_h\cup T)$, and $\overline C_z$ equals the closure
  $\overline{C_z}$ of $C_z$ in $\mathbb{P}^n$ for any $z\in Z$. Moreover, $\overline C-C=B=B_1\cup B_2$, where 
$$B_1=\{((0:\pm 1:\cdots: \pm 1), z)\ |\ z\in Z\},$$
  $$B_2=\Biggl\{\left( \left(0:0:\pm\sqrt{\frac{z_1-z_2}{z_1}}:
  \pm\sqrt{\frac{z_1-z_3}{z_1}}:\cdots :
  \pm\sqrt{\frac{z_1-z_{n-1}}{z_1}}:1\right), z\right)\ |\ z\in
  Z\Biggr\}.$$
\end{lemma}
\begin{proof}
Let $C'={\mathcal Z}(S_h\cup T)$. Then obviously $C'\cap (\mathbb{C}^n\times Z)=C$ and $\overline C \subseteq C'$.

We need to study the points on $C'$ that are at infinity (that is, those satisfying $y_0=0$).
Note that, no point in $C'$ satisfies $y_0=y_n=0$, otherwise $y_1y_i^2=0$ for $1\le i\le n-1$ and 
$y_i^2z_j=y_j^2z_i$ for $1\le i<j\le n-1$. Consequently, $y_i=0$ for all $0\le i\le n$, which is not possible. 
So we can make the following change of coordinates. 
Consider the open subset defined by $y_n\neq 0$, and
put $x_i=\frac{y_i}{y_n}$, for $i=0,1,\dots, n$. Then $S_h\cup T$
takes the following form in this new coordinate system.
$$\aligned
&\widetilde {S_h\cup T}=\{z_ix_0^3-x_1(x_i^2-1), (x_i^2-1)z_j-(x_j^2-1)z_i\
\text{for}\ i,j=1,2,\dots, n-1\}\\
&=\{z_ix_0^3-x_1x_i^2+x_1 \, (1\le i\le n-1), x_i^2z_j-z_j-x_j^2z_i+z_i \,
(1\le i<j\le n-1) \}.
\endaligned
$$
The solutions of the last set of equations, when specifying $x_0=0$, is obviously $B_1\cup B_2$. Therefore $C'-C=B_1\cup B_2$. 

It is obvious that $C'_z\supseteq \overline C_z\supseteq \overline{C_z}$. To prove that
both ``$\supseteq$'' are equalities, it suffices to show that any point
$p=(0,p_1,\dots,p_{n-1})\in C'_z-C_z=B_1^z\cup B_2^z$ (where $B_i^z=B_i\cap\pi^{-1}(z)$ for $i=1,2$)
is the endpoint of a path in $C_z$. Indeed, if $p\in B_1^z$, for $0\le t<\epsilon$
(where $\epsilon>0$ is sufficiently small), we let $x_0(t)=t$, let $x_1(t)$ be a
solution of $z_1t^3-x_1^3+x_1=0$ that is continuous on $t$ and $x_1(0)=p_1$;  
then for $i=2,\dots,n-1$, let 
$x_i(t)$ be a solution of $x_i^2=z_it^3/x_1(t)+1$ that is continuous on $t$ and $x_i(0)=p_i$.
Then the path $(x_0(t),\dots,x_{n-1}(t))$ ($0\le t<\epsilon$) works as needed. 
Similarly, 
if $p\in B_2^z$, we let $x_0(t)=t$, let $x_1(t)$ be a solution of $z_1t^3-x_1^3+x_1=0$ that is continuous on $t$ and $x_1(0)=0$;  
then for $i=2,\dots,n-1$, let 
$x_i(t)$ be a solution of $x_i^2=(x_1(t)^2-1)z_i/z_1+1$ that is continuous on $t$
and $x_i(0)=p_i$. Then the path $(x_0(t),\dots,x_{n-1}(t))$ ($0\le t<\epsilon$) works as needed. So $C'_z= \overline C_z = \overline{C_z}$.

It follows from $C'_z= \overline C_z$ (for all $z\in Z$) that $C'=\overline C$.
Thus $\overline C={\mathcal Z}(S_h\cup T)$ and $\overline C-C=B_1\cup B_2$. 
\end{proof}

  \begin{rem}{\rm 
We remark here that, using the Macaulay2 program, $\widetilde{S_h\cup T}$ can be obtained 
by computing the primary
decomposition of the ideal $\widetilde S_h$, which is $S_h$ in
the $x$-coordinate system. In another way, 
taking the homogenization of a Gr\"{o}bner basis of $S$, also gives 
the polynomials $\widetilde {S_h\cup T}$ to describe $\overline
C$, near infinity.}\end{rem}

  We have now enough information about the points in $\overline
  C$, to proceed to prove Proposition \ref{fibration}.

  We require the following lemmas.
 
\begin{lemma}\label{surjective}
  $\pi|_C:C\to Z$ is a surjective submersion.  As a consequence, $C_z$ 
is a smooth curve for all $z\in Z$.
  \end{lemma}

\begin{lemma}\label{submersion}  $\overline{C}$ is a smooth complex manifold of dimension $n$,
  and  $\pi|_{\overline{C}}: \overline{C}\to Z$ is a proper surjective submersion. As a consequence, ${\overline C}_z$ 
is a smooth projective curve for all $z\in Z$.\end{lemma}

\begin{lemma}\label{smooth} $B=\overline{C}-C$ is smooth and  $\pi|_B:B\to Z$ is a submersion.\end{lemma}

\begin{lemma} \label{connected}  The curve $\overline C_z$ is  connected and has
  genus $(3n-6)2^{n-3}+1$. As a consequence, $C_z$ is a connected curve of genus $(3n-6)2^{n-3}+1$ with $3\cdot 2^{n-2}$ punctures. 
  \end{lemma}

To prove Lemma \ref{surjective}, we need the following  
calculation of the determinant of a particular type of matrices.

\begin{slemma}\label{set} Consider the following $(n-1)\times n$
    matrix with complex entries. 
    $$
  M=\begin{pmatrix}b_0&0&0&\cdots&0&\cdots &0&c_0\\
    b_1&a_1&0&\cdots&0&\cdots &0&c_1\\
    b_2&0&a_2&\cdots&0&\cdots&0&c_2\\
    \vdots&&&\ddots&&\ddots&&\vdots\\
    b_{i-1}&0&0&\cdots&a_{i-1}&\cdots&0&c_{i-1}\\
    \vdots&&&\ddots&&\ddots&&\vdots\\
     b_{n-2}&0&0&\cdots&0&\cdots&a_{n-2}&c_{n-2}
   \end{pmatrix}$$
   Let $M'$ be the matrix obtained from $M$ after removing the
   $i$-th column for $2\leq i\leq n-1$. Then the determinant of $M'$
   is given by the following formula.
   $$|M'|=(-1)^{n+i-1}a_1\cdots a_{i-2}a_i\cdots a_{n-2}(b_0c_{i-1}-c_0b_{i-1})$$
 \end{slemma}
 
\begin{proof}
     Note that $M'$ is the following $(n-1)\times (n-1)$ matrix.
     $$
  M'=\begin{pmatrix}b_0&0&0&\cdots&0&0&\cdots &0&c_0\\
    b_1&a_1&0&\cdots&0&0&\cdots &0&c_1\\
    b_2&0&a_2&\cdots&0&0&\cdots&0&c_2\\
    \vdots&&&\ddots&&&\ddots&&\vdots\\
    b_{i-2}&0&0&\cdots&a_{i-2}&0&\cdots&0&c_{i-2}\\
    b_{i-1}&0&0&\cdots&0&0&\cdots&0&c_{i-1}\\
    b_i&0&0&\cdots&0&a_i&\cdots&0&c_i\\
    \vdots&&&\ddots&&&\ddots&&\vdots\\
     b_{n-2}&0&0&\cdots&0&0&\cdots&a_{n-2}&c_{n-2}
   \end{pmatrix}$$
Moving row $i$ to row 2, and moving column $n-1$ to column 2, we get
     $$
  M''=\begin{pmatrix}b_0&c_0&0&0&\cdots&&&\cdots &0\\
    b_{i-1}&c_{i-1}&0&0&\cdots&&&\cdots &0\\
    b_1&c_1&a_1&0&\cdots&&&\cdots &0\\
    b_2&c_2&0&a_2&\cdots&&&\cdots&0\\
    \vdots&&&&\ddots&&&&\vdots\\
    b_{i-2}&c_{i-2}&0&0&\cdots&a_{i-2}&0&\cdots&0\\
    b_i&c_i&0&0&\cdots&0&a_i&\cdots&0\\
    \vdots&&&&&&&\ddots&\vdots\\
     b_{n-2}&c_{n-2}&0&0&\cdots&0&0&\cdots&a_{n-2}
   \end{pmatrix}$$
  It follows that 
  $$|M'|=(-1)^{(i-2)+(n-1-2)}|M''|=(-1)^{n+i-1}a_1\cdots a_{i-2}a_i\cdots a_{n-2}(b_0c_{i-1}-c_0b_{i-1}).$$
\end{proof}

\begin{proof}[Proof of Lemma \ref{surjective}]
  For any $z\in Z$, we have constructed paths contained in $C_z$ in the proof of
  Lemma \ref{lemma B12}; so $C_z\neq\emptyset$. This shows surjectivity. 

  $\pi|_C$ is a submersion follows from the fact that the following
  Jacobian matrix has full rank $n-1$. 
  $$
  M_1=\begin{pmatrix}3y_1^2-y_n^2&0&0&\cdots &0&-2y_1y_n\\
    y_2^2-y_n^2&2y_1y_2&0&\cdots &0&-2y_1y_n\\
    y_3^2-y_n^2&0&2y_1y_3&\cdots&0&-2y_1y_n\\
    \vdots&&&\ddots&&\vdots\\
    y_{n-1}^2-y_n^2&0&0&\cdots&2y_1y_{n-1}&-2y_1y_n

    \end{pmatrix}$$

    \noindent
    {\bf Claim.} ${\rm rank} M_1=n-1$.

    \noindent
    {\it Proof of claim.} First note that $y_1\neq 0$. We consider now
    two cases.

    {\bf Case 1.} $y_n=0$. Then since $y_i\neq y_j$ for all $i\neq j$,
    we get $y_k\neq 0$ for all $k=2,\dots n-1$. Hence  the $(n-1)\times (n-1)$ submatrix obtained by removing the last column
    of $M_1$ has determinant $3\cdot 2^{n-2}y_1^{n}y_2\dots y_{n-1}\neq
    0$.
    Hence ${\rm rank} M_1=n-1$. 

    {\bf Case 2.} $y_n\neq 0$. This case is divided into two
    subcases.

    {\bf Subcase 2.1.} $y_i\neq 0$ for all $i=2,\dots, n-1$. The
    $(n-1)\times (n-1)$ submatrix obtained by removing the first column of $M_1$ has
    determinant $\pm2^{n-1}y_1^{n-2}y_2\cdots y_{n}\neq0$.  
  Therefore, $M_1$ has rank $n-1$.

  {\bf Subcase 2.2.} $y_{i_0}=0$ for some $2\le i_0\le n-1$. Then, clearly
  $y_k\neq 0$ for all $k\neq i_0$, since $y_i\neq y_j$ for $i\neq j$.  By Sublemma \ref{set}, the
      $(n-1)\times (n-1)$ submatrix obtained by removing $i_0$-th column of $M_1$ has determinant 
  $$\aligned
  &\quad (-1)^{n+i_0-1}(2y_1y_2)\cdots (2y_1y_{i_0-1})(2y_1y_{i_0+1})\cdots (2y_1y_{n-1})\\
  &\quad\quad\quad\quad\quad \cdot \big( (3y_1^2-y_n^2)(-2y_1y_n)-(-2y_1y_n)(-y_n^2)\big)\\
  &=(-1)^{n+i_0-1}3\cdot 2^{n-2}y_1^ny_2\cdots y_{i_0-1}y_{i_0+1}\cdots y_n\neq0
  \endaligned
  $$
Therefore, rank 
  of $M_1$ is $n-1$, in this case also.

  This proves our claim.
  Therefore, $\pi|_C$ is a submersion. The dimension of a fiber is $\dim C-\dim Z=1$, so $C_z$ is a smooth curve for any $z\in Z$. 
\end{proof}

\begin{proof}[Proof of Lemma \ref{submersion}]
The properness of $\pi|_{\overline{C}}$ follows from the facts that 
  $\overline C$ is closed and $\pi$ is proper.

  To show that  $\overline C$ is smooth, it suffices to show that for
  $p=(0,p_1,\dots,p_{n-1})\in B=B_1\cup B_2$, the Zariski tangent space $T_p\overline C$
  has dimension $n$ (\cite{Har77}, Chapter I, Exercise 5.10).
  Note that the Zariski tangent space is the kernel of the Jacobian $J$ of polynomials in $\widetilde
{S_h\cup T}$ and variables $x_0,\dots,x_{n-1},z_1,\dots,z_{n-1}$:
$$
\begin{pmatrix}
3z_1x_0^2&-3x_1^2+1&0&\cdots&0&x_0^3&0&\cdots&0\\
3z_2x_0^2&-x_2^2+1&-2x_1x_2&\cdots&0&0&x_0^3&\cdots&0\\
\vdots&&&\ddots&&&&\ddots\\
3z_{n-1}x_0^2&-x_{n-1}^2+1&0&\cdots &-2x_1x_{n-1}&0&0&\cdots&x_0^3\\
0&2z_jx_i&-2z_ix_j&&&-(x_j^2-1)&x_i^2-1\\
\end{pmatrix}
$$
where the last row should be understood as $\binom{n-1}{2}$ rows: that for
any $1\le i<j\le n-1$, there is a row with all 0 entries except the entries 
$2z_jx_i$, $-2z_ix_j$, $-(x_j^2-1)$, $x_i^2-1$ at places corresponding to variables $x_i,x_j, z_i, z_j$ respectively. 

If $p\in B_1$, it is easy to see that $\ker J$ is of the form $(*,0^{n-1},*^{n-1})$
where $*$ is an arbitrary number in $\mathbb{C}$, so $\dim(\ker J)=n$; moreover,
it projects surjectively to the last $(n-1)$ coordinates, so $\pi|_{\overline{C}}$ is a submersion at $p$.   

If $p\in B_2$, it is also straightforward to check that the kernel of $J$ consists of the following vectors: 
$$\Bigg(*,0,\frac{z_2u_1-z_1u_2}{2p_2z_1^2},\frac{z_3u_1-z_1u_3}{2p_3z_1^2},\dots,
\frac{z_{n-1}u_1-z_1u_{n-1}}{2p_{n-1}z_1^2},u_1,\dots,u_{n-1}
\Bigg)$$ 
where $*,u_1,\dots,u_{n-1}$ are arbitrary numbers in $\mathbb{C}$, so $\dim(\ker J)=n$;
moreover, it projects surjectively to the last $(n-1)$ coordinates, so $\pi|_{\overline{C}}$ is still a submersion at $p$.   

So we have proved that $\overline{C}$ is smooth, and $\pi|_{\overline{C}}$ is a submersion.
It follows that $\overline{C}_z$ is a smooth projective curve for all $z\in Z$.  
\end{proof}

\begin{proof}[Proof of Lemma \ref{smooth}]
We need to show that $B$ is smooth and $\pi|_B$ is a submersion. At the 
points of $B_1$, both these assertions are obvious.

Next, note that, $B_2$ is the zero set in ${\Bbb P}^n\times Z$ of the following polynomials.
$$\{x_0, x_1, x_2^2z_1-z_1^2+z_2^2,\dots, x_i^2z_1-z_1^2+z_i^2,\dots, x_{n-1}^2z_1-z_1^2+z_{n-1}^2\}.$$
Let $p=(0,p_1,\dots,p_{n-1})\in B_2$. 
The Jacobean matrix of the above polynomials with respect
to $x_0$, $x_1$,$\dots$, $x_{n-1}$, $z_1$, $\dots$, $z_{n-1}$ is the following:
$$J=\begin{pmatrix}1&0&0&\cdots&0&0&0&\cdots&0\\
  0&1&0&\cdots&0&0&0&\cdots&0\\
  0&0&2x_2z_1&\cdots&0&x_2^2-2z_1&2z_2&\cdots&0\\
\vdots&&&\ddots&&&&\ddots&\vdots\\
0&0&0&\cdots&2x_{n-1}z_1&x_{n-1}^2-2z_1&0&\cdots&2z_{n-1}
  \end{pmatrix}$$
the kernel of which consists of the following vectors:
$$\Bigg(0,0,\frac{(2z_1-p_2^2)u_1-2z_2u_2}{2p_2z_1},
\dots,
\frac{(2z_1-p_{n-1}^2)u_1-2z_{n-1}u_{n-1}}{2p_{n-1}z_1},
u_1,\dots,u_{n-1}
\Bigg)$$ 
where $u_1,\dots,u_{n-1}$ are arbitrary numbers in $\mathbb{C}$, so $\dim T_pB_2=\dim(\ker J)=n-1$ and hence
$B_2$ is smooth. Moreover, it projects surjectively to the last $(n-1)$ coordinates, so $\pi|_{B}$ is a submersion.  
\end{proof}

  \begin{proof}[Proof of Lemma \ref{connected}]
Fix $z\in Z$. Consider the projection map
$$\varphi: C_z\to \mathbb{C}^*, \quad (y_1,\dots,y_n)\mapsto y_1$$
The map $\varphi$ is generically $2^{n-1}$ to 1; the fiber $\varphi^{-1}(y_1)$ consists of points 
$$\Bigg(y_1,\pm\sqrt{y_1^2-\frac{z_1-z_2}{y_1}},\dots,
\pm\sqrt{y_1^2-\frac{z_1-z_{n-1}}{y_1}},
\pm\sqrt{y_1^2-\frac{z_1}{y_1}}\Bigg). 
$$
It suffices to show the points in each fiber are connected in $C_z$. 

Given a point $p$ in the fiber $\varphi^{-1}(y_1)$, construct a path
$\gamma\subset \mathbb{C}^*$ as follows: start from the point $y_1=\varphi(p)$
and go along a path $\alpha$ to a small neighborhood of $a=\sqrt[3]{z_1}$
(note that $a$ is not unique), go along a small circle $\beta$ around $a$,
then follow $\alpha^{-1}$ back to $y_1$ (so $\gamma$ is the concatenation
of $\alpha, \beta, \alpha^{-1}$).  Choose $\gamma$ carefully to avoid ramification points of $\varphi$. 
Now lift the path $\gamma$ to a path in $C_z$ starting at $p$. Then $\gamma$
will end at the point $(y_1,\dots,y_{n-1},-y_n)$. To see this, write $y_1=a+t$, then 
\begin{equation}\label{sqrt t}
\aligned
\sqrt{y_1^2-\frac{z_1}{y_1}}
&=\sqrt{(a+t)^2-\frac{a^3}{a+t}}=\sqrt{a^2+2at+t^2-a^2(1-\frac{t}{a}+\frac{t^2}{a^2}-\cdots)}\\
&=\sqrt{3at+O(t^3)}\approx \sqrt{3at}
\endaligned
\end{equation} so when $y_1$ goes around a circle $\beta$ around $a$,
$\sqrt{y_1^2-\frac{z_1}{y_1}}$ is multiplied by $-1$, therefore, $y_n$
becomes $-y_n$. On the other hand, it is easy to see that for $1\le i\le n-1$, $y_i$ does not change.

Similar to the above construction, but let $a$ be $\sqrt[3]{z_1-z_i}$, we can
obtain a path connecting $p$ with $(y_1,\dots,-y_{i},\dots, y_n)$, for any $1\le i\le n-1$. 

Therefore $C_z$ is connected. As a consequence, $\overline{C}_z$ is connected. 

\smallskip

Next, we compute the genus of $\overline{C}_z$ using the Riemann-Hurwitz formula.
Consider $\overline{\varphi}:\overline{C}_z\to\mathbb{P}_1$, the natural extension of $\varphi$. 
Then 
$$2g(\overline{C}_z)-2=2^{n-1}(2g(\mathbb{P}^1-2)+\sum (e_P-1)$$
We compute the ramification index at each possible ramification point. 

At $a=(\sqrt[3]{z_1}:1)$, the map $\overline{\varphi}$ can be locally written
as $t\mapsto t^2$ (see \eqref{sqrt t}), so the contribution to $\sum (e_P-1)$
is $2-1=1$. Since there are 3 choices for $\sqrt[3]{z_1}$ and $2^{n-2}$
possible choices for the signs: $(y_1,\pm y_2,\dots,\pm y_{n-1},0)$
The total contribution to $\sum (e_P-1)$ is $3\cdot 2^{n-2}$.

At $a=(\sqrt[3]{z_1-z_i}:1)$ ($2\le i\le n-1$), by a similar computation,
the contribution is  $3\cdot 2^{n-2}$ for each $i$. 

At $(0:1)$, $\overline{\varphi}^{-1}((0:1))=B_2^z$. For each $p\in B_2^z$, we parametrize $x_0=t$,
then solve $x_1$ from  $z_1t^3-x_1^3+x_1=0$, we get $x_1\approx -z_1t^3-z_1^3t^9$;   
then for $i=2,\dots,n-1$, let 
$x_i(t)$ be a solution of $x_i^2=z_it^3/x_1(t)+1$, so
$x_i\approx \sqrt{1-\frac{z_i}{z_1}(1-z_1^2t^6)}$. (This computation shows that
the curve is indeed analytically parametrized  by $t$.) The map $\overline{\varphi}$
at $p$ is locally $t\mapsto y_1/y_0=x_1/x_0\approx -z_1t^2$, or equivalently,
locally $t\mapsto t^2$. So the contribution is $|B_2^z|=2^{n-2}$. 

At $\infty:=(1:0)$, $\overline{\varphi}^{-1}(\infty)=B_1^z$. For each
$p=(0,p_1,\dots,p_{n-1})\in B_1^z$, we parametrize $x_0=t$,  then solve
$x_1$ from  $z_1t^3-x_1^3+x_1=0$, we get $x_1\approx p_1+\frac{z_1}{2}t^3$;   
then for $i=2,\dots,n-1$, let 
$x_i(t)$ be a solution of $x_i^2=z_it^3/x_1(t)+1$, so
$x_i\approx \sqrt{1-\frac{z_i}{z_1}(p_1-\frac{z_1}{2}t^3)}$. The map $\overline{\varphi}$
at $p$ is locally $t\mapsto y_0/y_1=x_0/x_1\approx t/p_1$, or equivalently, locally
an isomorphism. So there is no ramification at $p$.

To summarize in a table: 

\smallskip

\begin{tabular}{|c | c | c | c|} 
 \hline
 pts. in $\mathbb{P}^1$ & ramif. index  & \#of such points & contrib. to $\sum (e_P-1)$  \\ [0.5ex] 
 \hline
 $\sqrt[3]{z_1}$ and $\sqrt[3]{z_1-z_i}$ & $2$ & $3(n-1) 2^{n-2}$ & $3(n-1)2^{n-2}$ \\ 
 $B_2^z$ & $2$ & $2^{n-2}$ & $2^{n-2}$ \\
 $B_1^z$ & $1$ & $2^{n-1}$ & $0$ \\
 \hline
\end{tabular}
\smallskip

So
$2g(\overline{C}_z)-2=2^{n-1}(-2)+3(n-1)2^{n-2}+2^{n-2}=(3n-6)2^{n-2}$, 
then  $g(\overline{C}_z)=(3n-6)2^{n-3}+1$.  

As a consequence,
 $C_z$ is a connected curve of genus $(3n-6)2^{n-3}+1$ with $|B_1^z\cup B_2^z|=3\cdot 2^{n-2}$ punctures. 
  \end{proof}
  Note that the connectedness in Lemma \ref{connected} can also be proved using
  the idea of \cite[Lemma 4.6]{AR19}; the detail of which we shall skip. 
  
Now we are in a position to prove Proposition \ref{fibration}.

  \begin{proof}[Proof of Proposition \ref{fibration}]
    Note that $f=\pi|_C\circ \eta$. Since $\eta:Y\to C$ is a
  homeomorphism, $f$ is a locally trivial fibration if and only if so is $\pi|_C$.

  We proceed to prove that $\pi|_C$ is a fibration using Thom's
  First Isotopy Lemma.

  Since, $C$ is open and $B$ is its boundary, $\{C, B\}$ is a Whitney
  stratification of $\overline C$.
  Now, Lemmas \ref{surjective}, \ref{submersion}, \ref{smooth} show
  that the remaining conditions of the Thom's First Isotopy Lemma are satisfied
  for $\pi:{\Bbb P}^n\times Z\to Z$ and $\overline C$. Hence
  $\pi|_{\overline C}$ is a locally trivial fibration. Therefore, so
  is $\pi|_C$ also. Consequently, $f$ is a fibration with connected
  non-compact fibers.

  Next, we apply the Fadell-Neuwirth Fibration Theorem for $X={\Bbb C}^*$.
  By an induction argument on $n$ and using the long exact sequence of homotopy
  groups induced by a
  fibration, we get that $Z$ is aspherical and $\pi_1(Z, z)$ is poly-free.
  Consequently, again using the long exact sequence of homotopy groups induced by $f$, and
  using the fact that $Z$ is aspherical, we get the following short exact sequence.

  \centerline{\xymatrix {1\ar[r]& \pi_1(f^{-1}(z), y)\ar[r]& \pi_1(Y,y)\ar[r]^{f_*}& \pi_1(Z,z)\ar[r] &1.}}
  
  Therefore, 
  $\pi_1(Y, y)$ is also poly-free, since the fiber $f^{-1}(z)$ (homeomorphic to $C_z$) is a
  punctured, connected $2$-manifold by Lemma \ref{connected}.
  
  This completes the proof of Proposition \ref{fibration}.
\end{proof}

\newpage
\bibliographystyle{plain}

\begin{thebibliography}{60}

  \bibitem{All02}
  D. Allcock,
  \newblock Braid pictures of Artin groups,
\newblock {\em Trans. Amer. Math. Soc.} 354 (2002), no. 9, 
3455-3474.

\bibitem{ADR20}
N. Amend, P. Deligne and G. R\"{o}hrle,
\newblock On the $K(\pi, 1)$-problem for restrictions of complex
reflection arrangements,
\newblock {\em Compos. Math.} 156 (2020), no. 3, 526-532.

\bibitem{AR19}
N. Amend and G. R\"{o}hrle,
\newblock The topology of arrangements of ideal type,
\newblock {\em Algebr. Geom. Topol.} 19 (2019), no.3, 1341-1358.

  \bibitem{MB}
  M. Bestvina,
  \newblock Non-positively curved aspects of Artin groups of finite type,
  \newblock {\em Geom. Topol.} 3 (1999), 269-302.
  
  \bibitem{BFW23}
    M. Bestvina, K. Fujiwara and D. Wigglesworth,
    \newblock The Farrell-Jones conjecture for hyperbolic-by-cyclic groups.(English
    summary),
    \newblock {\em Int. Math. Res. Not. IMRN} (2023), no.7, 5887-5904.

    \bibitem{BMP}
  R. Blasco-Garc\'{i}a, C. Mart\'{i}nez-P\'{e}rez and L. Paris,
  \newblock Poly-freeness of even Artin groups of FC type,
  \newblock {\em Groups Geom. Dyn.} 13 (2019), no. 1, 309-325.

    \bibitem{Bri1}
E. Brieskorn, 
\newblock Die Fundamentalgruppe des Raumes der regul\"{a}ren Orbits einer
endlichen komplexen Spiegelungsgruppe, 
\newblock {\em Invent. Math.} 12 (1971), 57-61.

\bibitem{Br}
E. Brieskorn,
\newblock Sur les groupes de tresses [d'apr\`{e}s V.I. Arnol'd]
\newblock S\'{e}minaire Bourbaki, 24\`{e}me ann\'{e}e (1971/1972), 
Exp. No. 401, pp.21-44. Lecture Notes in Math., Vol 317, Springer, 
Berlin, 1973.

\bibitem{CMS10}
  F. Callegaro, D. Moroni and M. Salvetti,
  \newblock The $K(\pi, 1)$ problem for the affine Artin group of
  type $\widetilde B_n$ and its cohomology.
  \newblock {\em J. Eur. Math. Soc. (JEMS)} 12 (2010), 1-22.

  \bibitem{Cox}
H. S. M. Coxeter,
\newblock The complete enumeration of finite groups of the form $R^2_i=(R_iR_j)^{k_{ij}}=1$,
\newblock {\em Jour. London Math. Soc.} 10 (1935), 21-25.

\bibitem{Del}
P. Deligne,
\newblock Les immeubles des groupes de tresses g\'{e}n\'{e}ralis\'{e}s,
\newblock {\em Invent. Math.} 17 (1972), 273-302.

\bibitem{DS}
  A. Dranishnikov and J. Smith,
  \newblock Asymptotic Dimension of Discrete Groups,
  \newblock {\em Fund. Math.}, 189 (2006), 27-34.
  
\bibitem{FN62}
E. Fadell and L. Neuwirth,
\newblock Configuration spaces,
\newblock  {\em Math. Scand.} 10 (1962), 111-118.

\bibitem{JF23}
 J. Flechsig,
 \newblock Braid groups and mapping class groups for 2-orbifolds,
 \newblock {\em Bull. Sci. Math.} 204 (2025), Paper No. 103705, 61 pp.
 
 \bibitem{Har77}
 R. Hartshorne,
 \newblock Algebraic Geometry,
\newblock Grad. Texts in Math., No. 52
Springer-Verlag, New York-Heidelberg, 1977, xvi+496 pp.

\bibitem{Hum}
J. E. Humphreys,
\newblock Reflection groups and Coxeter groups,
\newblock Cambridge Studies in Advanced Mathematics 29, 
\newblock Cambridge University Press.

\bibitem{Mat71}
  J. Mather,
  \newblock Stratifications and mapping in: Dynamical systems
  (Proc. Sympos., Univ. Bahia, Salvador, 1971)(1973),
  195-232.
  \newblock Academic Press, Inc. [Harcourt Brace Jovanovich, Publishers], New York-London, 1973.

\bibitem{Ngu83}
  V.D. Nguy$\hat{\text{e}}\tilde{\text{n}}$,
  \newblock The fundamental groups of the spaces of regular orbits of
  the affine Weyl groups,
  \newblock {\em Topology} 22 (1983), 425-435.
  
\bibitem{Rou22}
  S. K. Roushon,
  \newblock A certain structure of Artin groups and the isomorphism conjecture,
  \newblock{\em Canad. J. Math.} 73 (2021), no. 4, 1153-1170.
  \newblock Corrigendum: A certain structure of Artin groups and the
  isomorphism conjecture, {\em Canad. J. of Math.}.
  \newblock Published online 2024:1-2. doi:10.4153/S0008414X24000191
  
\bibitem{Rou20}
  S. K. Roushon,
  \newblock Configuration Lie groupoids and orbifold braid groups, 
  \newblock {\em Bull. Sci. Math.} 171 (2021), Paper No. 103028, 35 pp.
  \newblock https://doi.org/10.1016/j.bulsci.2021.103028

\bibitem{Tho69}
  R. Thom,
  \newblock Ensembles et morphismes stratifi\'{e}s,
  \newblock {\em Bull. Amer. Math. Soc.} 75 (1969), 240-284.

  \bibitem{W}
  X. Wu,
  \newblock Poly-freeness of Artin groups and the Farrell-Jones Conjecture,
  \newblock {\em J. Group Theory} 25 (2022), no.1, 11-24.
\end{thebibliography}
\ifx\undefined\bysame
\newcommand{\bysame}{\leavevmode\hbox to3em{\hrulefill}\,}
\fi

\end{document}